\documentclass{amsart}

\usepackage{amssymb, amsmath}
\usepackage{mathrsfs, color}
\usepackage{amscd}
\usepackage{verbatim}

\usepackage{enumerate}

\usepackage{hyperref}

\allowdisplaybreaks

\theoremstyle{plain}
\newtheorem{theorem}{Theorem}[section]
\theoremstyle{remark}
\newtheorem{remark}[theorem]{Remark}
\newtheorem{example}[theorem]{Example}

\theoremstyle{plain}
\newtheorem{corollary}[theorem]{Corollary}
\newtheorem{lemma}[theorem]{Lemma}
\newtheorem{proposition}[theorem]{Proposition}
\newtheorem{definition}[theorem]{Definition}

\numberwithin{equation}{section}

\def\R{{\mathbb R}}

\newcommand{\E}{{\mathbb E}}
\renewcommand{\P}{{\mathbb P}}
\newcommand{\F}{{\mathscr F}}

\newcommand{\g}{\gamma}

\renewcommand{\O}{\Omega}

\newcommand{\Dom}{{\mathsf D}}

\newcommand{\calL}{{\mathscr L}}
\newcommand{\n}{\Vert}
\newcommand{\one}{{{\bf 1}}}
\newcommand{\embed}{\hookrightarrow}
\newcommand{\s}{^*}
\newcommand{\lb}{\langle}
\newcommand{\rb}{\rangle}

\newcommand{\wt}{\widetilde}

\DeclareMathOperator*{\esssup}{ess\,sup}

\begin{document}

\author{Jan van Neerven}
\address{Delft Institute of Applied Mathematics\\
Delft University of Technology \\ P.O. Box 5031\\ 2600 GA Delft\\The
Netherlands} \email{J.M.A.M.vanNeerven@tudelft.nl}

\author{Mark Veraar}
\address{Delft Institute of Applied Mathematics\\
Delft University of Technology \\ P.O. Box 5031\\ 2600 GA Delft\\The
Netherlands} \email{M.C.Veraar@tudelft.nl}

\author{Lutz Weis}
\address{
Department of Mathematics\\ Karlsruhe Institute of Technology (KIT) \\ D-76128
Karls\-ruhe \\ Germany}
\email{Lutz.Weis@kit.edu}

\title[On the $R$-boundedness of stochastic convolution operators]{On the $R$-boundedness of stochastic \\ convolution operators}

\begin{abstract}
The $R$-boundedness of certain families of vector-valued stochastic convolution operators with scalar-valued
square integrable kernels is the
key ingredient in the recent proof of stochastic maximal $L^p$-regularity, $2<p<\infty$, for
certain classes of sectorial operators acting on spaces $X=L^q(\mu)$, $2\le q<\infty$. This
paper presents a systematic study of $R$-boundedness of such families.
Our main result generalises the afore-mentioned $R$-boundedness result to a larger
class of Banach lattices $X$ and relates it to the $\ell^{1}$-boundedness of an associated class of
deterministic convolution operators.
We also establish an intimate relationship between the $\ell^{1}$-boundedness
of these operators and the boundedness of the $X$-valued maximal function.
This analysis leads, quite surprisingly, to an example showing that $R$-boundedness
of stochastic convolution operators
fails in certain UMD Banach lattices with type $2$.
\end{abstract}

\thanks{The first named author is supported by VICI subsidy 639.033.604
of the Netherlands Organisation for Scientific Research (NWO)}



\date\today

\maketitle

\section{Introduction}

Maximal $L^p$-regularity is a tool of central importance in the theory of parabolic PDEs,
as it enables one to reduce the study of various classes of `complicated' non-linear PDEs
to a fixed point problem, e.g. by linearisation (see \cite{Are04, DHP, KuWe} and the references therein).
The extension of this circle of ideas to parabolic stochastic PDE required new ideas and was achieved
only recently in \cite{NVW12a}, where it was shown that if a sectorial operator $A$ admits a
bounded $H^\infty$-calculus of angle less than $\pi/2$ on a space $L^q(D,\mu)$, with $q\in [2,\infty)$
and $(D,\mu)$ a $\sigma$-finite measure space, then for all Hilbert spaces $H$ and adapted processes
$G\in L^p(\R_+\times\Omega;L^q(D,\mu;H))$
the stochastic convolution process
$$
U(t) = \int_0^t e^{-(t-s)A}G(s)\,dW_H(s), \quad t\ge 0,
$$ with respect to any cylindrical Brownian motion $W_H$ in  $H$,
is well-defined in $L^q(D,\mu)$, takes values in the fractional domain
$\Dom(A^{1/2})$ almost surely,
and satisfies, for $2<p<\infty$, the stochastic maximal $L^p$-regularity
estimate
\begin{equation}\label{eq:Apq}
\E \n A^{1/2} U\n_{L^p(\R_+;L^q(D,\mu))}^p \le C^p \E\n
G\n_{L^p(\R_+;L^q(D,\mu;H))}^p.
\end{equation}
Applciations to semilinear parabolic SPDEs were worked out subsequently in
\cite{NVW-SIAM}.
 By now, two proofs of the stochastic maximal $L^p$-regularity
theorem are available: the original one of \cite{NVW12a} based on $H^\infty$-calculus techniques combined with
the Poisson
formula for holomorphic functions on an open sector in the complex plane, and a second one
based on operator-valued $H^\infty$-calculus techniques \cite{NVW13}. Both proofs, however, critically depend upon
the $R$-boundedness of a suitable class of vector-valued
stochastic convolution operators with scalar-valued
kernels. For stochastic convolution operators taking values in a space $L^q(\mu)$ with
$2\le q<\infty$, the $R$-boundedness of this family has been derived in \cite{NVW12a}
as a consequence of the Fefferman-Stein theorem on the $L^p(L^q(\mu))$-boundedness of the
Hardy-Littlewood maximal function; it is for this reason that the theory, in its present state,
is essentially limited to SPDEs with state space $X=L^q(\mu)$.

The aim of this paper is to undertake a systematic analysis
of the $R$-boundedness properties of families of
stochastic convolution operators with scalar-valued square integrable
kernels $k$ taking values in an arbitrary
Banach lattice $X$. The main result asserts that
such a family is $R$-bounded if and only if the corresponding family of {\em deterministic}
convolution operators corresponding to the squared kernels $k^2$ is $\ell^{1}$-bounded.
The notion of $\ell^{s}$-boundedness (also called $R^s$-boundedness), $1\le s\le \infty$,
has been introduced in  \cite{Weis-maxreg} and was systematically studied
in \cite{KunUll, UllmannPhD}. For operators acting on
Banach lattices $X$ with finite cotype, $R$-boundedness is equivalent to $\ell^{2}$-boundedness. Moreover, in \cite{KVW} it is shown that this can only be true if $X$ has finite cotype.

Thus the problem of stochastic maximal $L^p$-regularity is reduced to the problem of
$\ell^{1}$-boundedness of suitable families of deterministic convolution operators with integrable
kernels. Our second main result establishes the $\ell^{1}$-boundedness of such operators
under the assumption that $X$ is a Banach lattice with type $2$ with the additional
property that the dual of its $2$-convexification has the so-called Hardy-Littlewood property, meaning
essentially that the Fefferman-Stein theorem holds for this space. A sufficient condition for the
latter is that the $2$-convexification is a UMD Banach function space.
In \cite[Theorem 3]{PoSuXu}, the same condition was shown to imply the the
$X$-valued Littlewood-Paley-Rubio de Francia property.

In Section \ref{sec:failure} we show that the Banach lattice $\ell^\infty(\ell^2) = (\ell^1(\ell^2))\s$ fails the
Hardy-Littlewood property (see Definition \ref{def:HL} below), and for this reason $X= \ell^2(\ell^4)$
(whose $2$-concavif\-ication equals $\ell^1(\ell^2)$) is a natural candidate
of a Banach lattice in which $R$-boundedness of $X$-valued stochastic convolution operators
might fail. In the final section of this
paper we establish our third main result, which turns this suspicion into a theorem.
The failure of $R$-boundedness of stochastic convolutions
in $\ell^2(\ell^4)$ is quite remarkable, as this space is a UMD Banach lattice with type $2$.

\section{Preliminaries}

Throughout this paper, all vector spaces are real.
In this preliminary section we collect some results that will be needed in the sequel.

\subsection{$R$-boundedness} (See \cite{DHP, KuWe}).
Let $X$ and $Y$ be real Banach
spaces and let $(r_n)_{n\ge 1}$ be a {\em Rademacher sequence} on a probability space
$(\Omega,\P)$, that is, a sequence of independent random variables $r_n:\Omega\to \{-1,1\}$
taking the values $\pm 1$ with probability $\frac12$.
A family $\mathscr{T}$ of bounded linear operators from $X$ to $Y$ is
called {\em $R$-bounded} if there exists a constant $C\ge 0$ such that for all
finite sequences $(T_n)_{n=1}^N$ in ${\mathscr
{T}}$ and $(x_n)_{n=1}^N$ in $X$ we have
\[ \E \Big\n \sum_{n=1}^N r_n T_n x_n\Big\n ^2
\le C^2\E \Big\n \sum_{n=1}^N r_n x_n\Big\n ^2.
\]
The least admissible constant $C$ is called the {\em $R$-bound} of $\mathscr
{T}$, notation $R(\mathscr{T})$.

\subsection{Spaces of radonifying operators}\label{subsec:rad}
(See \cite{NeeCMA}).
Let $H$ be a Hilbert space and $X$ a Banach space.
 For $h\in H$ and $x\in X$ we denote by $h\otimes x$ the rank one operator from $H$ to $X$ given by $h'\mapsto [h',h]x$.
Let $(\g_n)_{n\ge 1}$ be a Gaussian sequence defined on some probability space $(\O,\P)$.
The {\em $\g$-radonifying norm} of a finite rank operator
of the form $\sum_{n=1}^N h_n\otimes x_n$, where the vectors
 $h_1,\dots,h_N$ are orthonormal in $H$ and $x_1,\dots,x_N$ are taken from $X$, is defined by
\begin{align}\label{eq:gamma} \Big\n \sum_{n=1}^N h_n\otimes x_n \Big\n_{\g(H,X)}^2 :=
\E \Big\n  \sum_{n=1}^N \g_n x_n \Big\n^2.\end{align}
The invariance of standard Gaussians vectors in $\R^n$ under orthogonal transformations
easily implies that this is well defined.
The completion of the space $H\otimes X$ of all finite rank operators from $H$ into $X$
with respect to
the norm $\n \cdot\n_{\gamma(H,X)}$ is denoted by $\g(H,X)$. This space is
continuously and contractively embedded in
$\calL(H,X)$. A bounded operator in $\calL(H,X)$ is said to be {\em $\g$-radonifying} if it
belongs to $\g(H,X)$.
If $H$ is separable, say with orthonormal basis $(h_n)_{n\ge 1}$, then an operator $T\in \calL(H,X)$
is $\gamma$-radonifying if and only if the sum
$\sum_{n\ge 1} \gamma_n Th_n $ converges in $L^2(\Omega;X)$, and in this case we have
$$\n T\n_{\gamma(H,X)}^2 = \E\Big\n  \sum_{n\ge 1} \gamma_n Th_n\Big\n^2.$$

The space
$\g(H,X)$ is an {\em operator ideal} in $\calL(H,X)$ in the sense that if
$S_1:\tilde{H}\to H$ and $ S_2:X\to \tilde{X}$
are bounded operators, then $T\in \g(H,X)$ implies
$S_2TS_1\in \g(\tilde{H},\tilde{X})$ and
\begin{equation}\label{eq:ideal}
 \n S_2TS_1\n_{ \g(\tilde{H},\tilde{X})} \leq \n S_2\n \n T\n_{\g(H,X)} \n S_1\n.
\end{equation}

Let $p\in [1,\infty)$ be given, let $(\Omega,\P)$ be a probability space,
and suppose that $W : L^2(\R_+;H) \to L^2(\Omega)$ is an $H$-cylindrical Brownian motion
(see Section \ref{sec:stoch} for the precise definition). Then the
stochastic integral $h\otimes x \mapsto Wh\otimes x$ extends to an isomorphic embedding
of $\g(L^2(\R_+;H),X)$ onto a closed subspace of $L^p(\O;X)$. This fact will be used in the proof of
Proposition \ref{prop:alphaplus}; a more detailed account of stochastic integration with respect to cylindrical Brownian motion will be given in Section \ref{sec:stoch}.

\begin{example}
If $X$ is a Hilbert space, then $\g(H,X)$ is isometrically
isomorphic to the Hilbert space of
Hilbert--Schmidt operators from $H$ to $X$. If $(S,\mu)$ is a $\sigma$-finite measure
space and $p\in [1,\infty)$, then $\gamma(H,L^p(S)) = L^p(S;H)$ with equivalent norms,
the isomorphism being given by associating to the function $f\in L^p(S;H)$ the mapping
$h\mapsto [f(\cdot),h]_H$ from $H$ to $L^p(S)$.
\end{example}

More generally we have (see \cite[Proposition 2.6]{NVW1}) :
\begin{proposition}[$\gamma$-Fubini isomorphism]\label{prop:g-Fub}
For any Banach space $X$ the mapping $h\otimes (f\otimes x) \mapsto f\otimes (h\otimes x)$
extends by linearity to an isomorphism of Banach spaces
$$ \gamma(H,L^p(\R^d;X)) \eqsim L^p(\R^d; \gamma(H,X)).$$
\end{proposition}

The next simple proposition extends \cite[Proposition 6.2]{MaaNee}. For $f\in L^p(S)$ and
$y\in Y$ we denote by $f\otimes y$ the function in $L^p(S;Y)$ defined by $(f\otimes y)(s):= f(s)y$.
By $I_H$ we denote the identity operator on $H$.

\begin{proposition}\label{prop:Rbdd-ext}
Let $\mathscr{T}$ be an $R$-bounded family of bounded linear operators from $X$ to $Y$
and let $H$ be a nonzero Hilbert space. Then the family
$I_H\otimes \mathscr{T} = \{I_H\otimes T: \ T \in \mathscr{T}\}$ is $R$-bounded
from $\gamma(H,X)$ to $\gamma(H,Y)$ and
$R(I_H\otimes \mathscr{T})= R(\mathscr{T})$.
\end{proposition}
\begin{proof}
Fix $T_1,\dots,T_N\in\mathscr{T}$ and $R_1,\dots,R_N\in\gamma(H,X)$.
Since each $R_n$ is the limit of at most countably many finite rank operators we may assume that $H$ is separable.
Let $(h_m)_{m\ge 1}$ be an orthonormal basis for $H$. Then,
\begin{align*}
 \E_r \Big\n \sum_{n=1}^N r_n (I_H\otimes T_n)R_n \Big\n_{\gamma(H,Y)}^2
 &
= \E_r \E_\g \Big\n\sum_{n=1}^N \sum_{m\ge 1} r_n \gamma_m T_nR_n h_m \Big\n ^2
\\ & \le R(\mathscr{T})^2 \E_r \E_\g \Big\n \sum_{n=1}^N\sum_{m\ge 1} r_n \gamma_m R_n h_m \Big\n ^2
\\ & = R(\mathscr{T})^2\E_r \Big\n \sum_{n=1}^N r_n R_n \Big\n_{\gamma(H,X)}^2.
\end{align*}
This proves the $R$-boundedness of $I_H\otimes \mathscr{T}$ along with the bound
$R(I_H\otimes \mathscr{T})\le R(\mathscr{T})$. The converse inequality is trivial.
\end{proof}

\subsection{Type and cotype}(See \cite{DJT, LeTa}).
A Banach space $X$ has {\em type $p\in [1,2]$} if there exists a constant
$C\ge 0$ such that for all $N\ge 1$ and all finite sequences $(x_n)_{n=1}^N$
in $X$ we have
$$
\Big(\E \Big\| \sum_{n=1}^N r_n x_n\Big\|^2\Big)^\frac12\le C
\Big(\sum_{n=1}^N \|x_n\|^p\Big)^\frac1p.
$$
The least admissible constant $C$ is called the {\em type $p$ constant} of $X$, notation
$T_p(X)$.
Similarly, $X$ has {\em cotype $q\in [2,\infty]$} if there exists a constant $C\ge 0$ such that for all
$N\ge 1$ and all finite sequences  $(x_n)_{n=1}^N$
in $X$ we have
$$
\Big(\sum_{n=1}^N \|x_n\|^q\Big)^\frac1q \le C\Big(\E \Big\| \sum_{n=1}^N
r_n\, x_n\Big\|^2\Big)^\frac12
$$
(with an obvious modification if $q=\infty$).
The least admissible constant $C$ is called the {\em cotype $q$ constant} of $X$, notation
$C_q(X)$.

Every Banach space has type $1$ and cotype $\infty$, Hilbert spaces have type $2$
and cotype $2$, and the spaces $L^p(S)$ have type $\min\{p,2\}$ and
cotype $\max\{p,2\}$ for $p\in [1,\infty)$. All UMD spaces have non-trivial type, i.e.,
type $p\in (1,2]$.

Spaces of type $2$ are of special importance to us for the following reason.

\begin{proposition}[\hbox{\cite{vNWe, RS}}]\label{prop:L2gamma}
Let $(S,\Sigma,\mu)$ be a $\sigma$-finite measure space, $H$ a non-zero Hilbert space,
and $X$ a Banach space.
\begin{enumerate}[\rm(1)]
 \item
$X$ has type $2$ if and only if
the mapping
$$f\otimes (h\otimes x) \mapsto (f\otimes h)\otimes x, \qquad f\in L^2(S), \ h\in H, \ x\in X, $$
extends to continuous embedding
$$
I: L^2(S;\g(H,X))\embed \g(L^2(S;H),X).
$$
In this case we have $\n I\n\le T_2(X)$.

\item $X$ has cotype $2$ if and only if
the mapping $$(f\otimes h)\otimes x \mapsto f\otimes (h\otimes x), \qquad f\in L^2(S), \
h\in H, \ x\in X, $$
extends to continuous embedding
$$
J: \g(L^2(S;H),X) \embed L^2(S;\g(H,X)).
$$
In this case we have $\n J\n \le C_2(X)$.
\end{enumerate}
\end{proposition}

\subsection{Double Rademacher sums}(See \cite{NeeWei08}).
Let $(r_{mn})_{m,n\geq 1}$ be a doubly indexed Rade\-macher
sequence on a probability space
$(\Omega,\P)$  and let $(r_n')_{n\geq 1}$ and $(r_m'')_{n\geq 1}$
be Rade\-macher sequences on independent probability spaces
$(\Omega',\P')$ and $(\Omega'',\P'')$ respectively.
\begin{definition}[See \cite{NeeWei08,Pialpha}] Let $X$ be a Banach space.
\begin{enumerate}[\rm (1)]

\item
 $X$ has {\em property $(\alpha^+)$}
 if there is a constant $C^+\ge 0$ such that
for all finite doubly-indexed sequences $(x_{mn})_{m,n=1}^{M,N}$ in $X$ we
have
\[
\begin{aligned}
 \Big(\E \Bigl\| \sum_{m=1}^M\sum_{n=1}^N r_{mn} x_{m n} \Bigr\|^2\Big)^{1/2}
\le C^+ \Big(\E'\E'' \Bigl\| \sum_{m=1}^M\sum_{n=1}^N r_{m}'r_n''x_{m n}
\Bigr\|^2\Big)^{1/2}.
\end{aligned}
\]

\item
 $X$ has {\em property $(\alpha^-)$} if there is a constant $C^-\ge 0$ such that
for all finite doubly-indexed sequences $(x_{mn})_{m,n=1}^{M,N}$ in $X$ we
have
\[
\begin{aligned}
 \Big(\E'\E'' \Bigl\| \sum_{m=1}^M\sum_{n=1}^N r_{m}'r_n''x_{m n}
 \Bigr\|^2\Big)^{1/2}
\le C^- \Big(\E \Bigl\| \sum_{m=1}^M\sum_{n=1}^N r_{mn} x_{m n}
\Bigr\|^2\Big)^{1/2}.
\end{aligned}
\]

\item $X$ has {\em property $(\alpha)$} if $X$ has property $(\alpha^+)$ and
$(\alpha^-)$.
\end{enumerate}
\end{definition}
Each of the properties  $(\alpha^+)$ and
$(\alpha^-)$ implies finite cotype, and conversely
every Banach lattice with finite cotype has property $(\alpha)$.
The space $c_0$ fails both $(\alpha^+)$ and
$(\alpha^-)$. For the Schatten class $C^p$ with
$p\in [1, \infty)$ one has the following results which follows from the proofs in \cite{PiX:96}:
\begin{itemize}
\item[\rm(i)] $C^p$ has property $(\alpha^+)$ if and only if $p\in [2, \infty)$
\item[\rm(ii)] $C^p$ has property $(\alpha^-)$ if and only if $p\in [1, 2]$.
\end{itemize}
In particular, $C^p$ has property $(\alpha)$ if and only if $p=2$.

Below we shall need part (1) of the following result.

\begin{proposition}[\mbox{\cite[Theorem 3.3]{NeeWei08}}]\label{prop:alpha} Let $H_1$ and $H_2$ be non-zero Hilbert spaces and denote by
$H_1\overline\otimes H_2$ their Hilbert space tensor product. The following assertions hold:
\begin{enumerate}[\em(1)]
 \item $X$ has property $(\alpha^+)$ if and only if the map $h_1\otimes (h_2\otimes x) \mapsto
(h_1\otimes h_2)\otimes$ extends to a bounded operator from
$\g(H_1,\g(H_2,X))$ into $ \g(H_1\overline\otimes H_2,X)$;

 \item $X$ has property $(\alpha^-)$ if and only if the map $(h_1\otimes h_2)\otimes \mapsto
h_1\otimes (h_2\otimes x)
$ extends to a bounded operator from $ \g(H_1\overline \otimes H_2,X)$ into
$\g(H_1,\g(H_2,X))$.
\end{enumerate}
\end{proposition}

The next result establishes a relation between the notions of type and cotype and the
properties $(\alpha^+)$ and $(\alpha^-)$.

\begin{proposition}\label{prop:typealpha}
Let $X$ be a Banach space.
\begin{enumerate}[(1)]
\item[\rm(1)] If $X$ has type $2$, then $X$ has property $(\alpha^+)$.
\item[\rm(2)] If $X$ has cotype $2$, then $X$ has property $(\alpha^-)$.
\end{enumerate}
\end{proposition}
\begin{proof}
(1): \
Since $(r_m'r_n'')_{m,n\ge 1}$ is an orthonormal system
in $L^2(\Omega'\times\Omega'')$,
part (1) is a consequence of \cite[Theorem 1.3]{DoSu}.
A more direct proof, which can be modified to give part (2) as well, runs as follows.

Set $h_{mn} := r_m'r_n''$.
Estimating Rade\-macher sums by Gaussian sums (see \cite[Proposition 12.11]{DJT} or \cite[Lemma 4.5]{LeTa}) and using
\eqref{eq:gamma} and Proposition \ref{prop:L2gamma}(1), we obtain
\begin{align*}  \E \Bigl\| \sum_{m=1}^M\sum_{n=1}^N r_{mn} x_{m n} \Bigr\|^2
& \le \tfrac12\pi \E \Bigl\| \sum_{m=1}^M\sum_{n=1}^N \gamma_{mn} x_{m n} \Bigr\|^2
\\ & = \tfrac12\pi \Bigl\| \sum_{m=1}^M\sum_{n=1}^N h_{mn} \otimes x_{m n}
\Bigr\|_{\gamma(L^2(\Omega'\times\Omega''),X)}^2
\\ & \le \tfrac12\pi (T_2(X))^2
\Bigl\| \sum_{m=1}^M\sum_{n=1}^N h_{mn} \otimes x_{m n} \Bigr\|_{L^2(\Omega'\times\Omega'';X)}^2
\\ & = \tfrac12\pi (T_2(X))^2
\E'\E'' \Bigl\| \sum_{m=1}^M\sum_{n=1}^N r_m'r_n'' x_{m n} \Bigr\|^2.
\end{align*}
This gives the result.

(2):\ This is proved in the same way, this time using Proposition \ref{prop:L2gamma}(2)
along with the fact that in the presence of finite cotype, Gaussian sums can be estimated
by Rademacher sums (see \cite[Theorem 12.27]{DJT} or \cite[Proposition 9.14]{LeTa}).
\end{proof}

\begin{lemma}\label{lem:squarefunction}
Let $X$ be a Banach function space with finite cotype and
let $(S,\Sigma,\mu)$ be a $\sigma$-finite measure space.
Let $G_n\in L^2(S;X)$,  $1\leq n\leq N$, be functions taking values in a finite-dimensional subspace of $X$.
Then for all $1\le p<\infty$ we have
\[\Big(\E\Big\|\sum_{n=1}^N r_n G_n\Big\|_{\gamma(L^2(S),X)}^p\Big)^{1/p}\eqsim_{p,X}
\Big\|\Big(\sum_{n=1}^N \int_S |G_n|^2\, d\mu\Big)^{1/2}\Big\| .\]
\end{lemma}

\begin{proof}
By the Kahane-Khintchine inequalities it suffices to consider $p=2$.
By \cite[Proposition 12.11 and Theorem 12.27]{DJT},
\[\E\Big\|\sum_{n=1}^N r_n G_n\Big\|_{\gamma(L^2(S),X)}^2\eqsim
\E\Big\|\sum_{n=1}^N \gamma_n G_n\Big\|_{\gamma(L^2(S),X)}^2 = \|G\|_{\gamma(\ell^2_N;\gamma(L^2(S),X))}^2.\]
Moreover, by Proposition \ref{prop:alpha},
$\gamma(\ell^2_N,\gamma(L^2(S),X))\eqsim \gamma(L^2(S)\otimes\ell^2_N,X))$ isomorphically. Now the result follows
from (the proof of) \cite[Corollary 2.10]{NW1}.
\end{proof}

\subsection{The UMD property and martingale type}(See \cite{Burk01, Pi75, RF}).
A Banach space $X$ is called a {\em UMD space} if for some $p\in (1, \infty)$ (equivalently,
for all $p\in (1,\infty)$; see \cite{Burk01}) there is a constant $\beta\ge 0$ such that for all
finite $X$-valued $L^p$-martingale difference sequences $(d_n)_{n=1}^N$
and sequence of signs $(\varepsilon_n)_{n=1}^N$ one has
\begin{equation}\label{eq:UMD}
 \E \Big\| \sum_{n=1}^N \varepsilon_n d_n\Big\|^p \le \beta^p  \E \Big\| \sum_{n=1}^N  d_n\Big\|^p.
\end{equation}
The least admissible constant in this definition is called the {\em UMD$_p$-constant} of $X$ and is
denoted by $\beta_{p,X}$. If $(r_n)_{n\ge 1}$ is a Rademacher sequence which is independent of $(d_n)_{n=1}^N$,
then \eqref{eq:UMD} and its counterpart applied to the martingales
$\varepsilon_n d_n$ easily imply the two-sided randomised inequality
\begin{equation}\label{eq:UMDrandom}
\frac{1}{\beta_{p,X}^p}\E\Big\| \sum_{n=1}^N d_n \Big\|^p\leq \E \Big\| \sum_{n=1}^N r_n d_n \Big\|^p
\end{equation}
where now $(r_n)_{n\ge 1}$ is a Rademacher sequence independent of $(d_n)_{n=1}^N$.

Examples of UMD spaces include Hilbert spaces and the Lebesgue spaces $L^p(S)$ for $1<p<\infty$.
Noting that every UMD space is reflexive, it follows that $L^\infty(S)$ and $L^1(S)$ are not UMD spaces.

Let $p\in [1,2]$.
 A Banach space $X$ has {\em martingale type $p$} if there exists a constant $\mu\ge 0$ such that for all
 finite $X$-valued martingale difference sequences $(d_n)_{n=1}^N$ we have
\begin{equation}\label{eq:martingaletype}
\E\Big\| \sum_{n=1}^N d_n\Big\|^p \le \mu^p\sum_{n=1}^N\E \| d_n\|^p.
\end{equation}
The least admissible constant in this definition is denoted by $\mu_{p,X}$.

Trivially, martingale type $p$ implies type $p$.
Hilbert spaces have martingale type $2$ and every Lebesgue space $L^p(S)$, $1\le p<\infty$, has martingale type $p\wedge 2$.
In fact we have the following equivalence (see \cite{Brz1}):

\begin{proposition}\label{prop:type-Mtype} Let $p\in [1,2]$.
\begin{enumerate}[\rm(1)]
 \item
 A UMD Banach space $X$ has martingale type $p$ if and only if it has type $p$.
\item
 A Banach lattice $X$ has martingale type $2$ if and only if it has type $2$.
\end{enumerate}
\end{proposition}
\begin{proof}
(1): \ Suppose that $X$ has type $p$ and let $(\wt r_n)_{n\geq 1}$ be a Rademacher sequence on another probability space $(\wt\Omega,\wt \P)$.
By \eqref{eq:UMDrandom} and Fubini's theorem,
\[
\E \Big\| \sum_{n=1}^N d_n\Big\|^p \leq \beta_{p,X}^p \E\wt \E\Big\| \sum_{n=1}^N \wt r_n d_n\Big\|^p
 \leq \beta_{p,X}^p \tau_{p,X}^p \E \sum_{n=1}^N \|d_n\|^p.
\]
It follows that $X$ has martingale type $p$.

(2): \  Suppose that $X$ has type $2$. By \cite[Theorem 1.f.17]{LiTz}, $X$ is $2$-convex and $q$-concave for some
$q<\infty$. By \cite[Theorem 1.f.1]{LiTz}, this implies that $X$ is $2$-smooth. Hence by \cite{Pi75},
$X$ has martingale type $2$.
\end{proof}

\section{$\ell^{s}$-Boundedness}

For Rademacher sums with values in a Banach lattice
$X$ with finite cotype we have the two-sided estimate
\begin{align}\label{eq:sqf}
\Big(\E \Big\n \sum_{n=1}^N r_n x_n\Big\n ^2\Big)^{1/2}
\eqsim \Big(\E \Big\n \sum_{n=1}^N \gamma_n x_n\Big\n ^2\Big)^{1/2}
\eqsim \Big\|\Big(\sum_{n=1}^N |x_n|^2\Big)^{1/2}\Big\|
\end{align}
with implied constants depending only on $X$ (see \cite[Proposition 12.11, Theorems 12.27 and 16.18]{DJT}).
The expression on the right-hand side acquires its meaning through the so-called Krivine calculus.
We refer to \cite[Section II.1.d]{LiTz} for a detailed exposition of this calculus.
In all applications below, $X$ is a Banach function
space and in this special case the expression can be defined in a pointwise sense in the obvious
manner. If $X$ and $Y$ are Banach lattices with finite cotype,
a family $\mathscr{T}$ of
bounded linear operators from $X$ to $Y$ is $R$-bounded if and only if there is a constant
$C\ge 0$ such that for all finite
sequences $(T_n)_{n=1}^N$ in $\mathscr{T}$ and $(x_n)_{n=1}^N$ in $X$ we have
$$
\Big\|\Big(\sum_{n=1}^N |T_n x_n|^2\Big)^{1/2}\Big\|  \leq C \Big\|\Big(\sum_{n=1}^N |x_n|^2\Big)^{1/2}\Big\| ,
$$
This motivates the following definition.

\begin{definition}\label{def:Rs}
Let $X$ and $Y$ be a Banach lattices and let $s\in [1, \infty]$.
A family of operators $\mathscr{T}\subseteq \calL(X,Y)$ is
called {\em $\ell^{s}$-bounded} if there is a constant $C\ge 0$ such that for all finite
sequences $(T_n)_{n=1}^N$ in $\mathscr{T}$ and $(x_n)_{n=1}^N$ in $X$ we have
$$
\Big\|\Big(\sum_{n=1}^N |T_n x_n|^s\Big)^{1/s}\Big\|  \leq C \Big\|\Big(\sum_{n=1}^N |x_n|^s\Big)^{1/s}\Big\| ,
$$
with the obvious modification if $s=\infty$.
\end{definition}

The least admissible constant $C$ in Definition \ref{def:Rs}
is called the {\em $\ell^{s}$-bound} of $\mathscr{T}$ and is
denoted by $R^{\ell^s}(\mathscr{T})$ and usually abbreviated as $R^{s}(\mathscr{T})$.

The notion of $\ell^{s}$-boundedness was introduced in \cite{Weis-maxreg} in the context of the
so-called (deterministic) maximal regularity problem; for a systematic treatment we refer the reader to
 \cite{KunUll, UllmannPhD}.

\begin{example}[\hbox{\cite[Remark 2.7]{KunUll}}]\label{ex:real}
Let $(S,\mu)$ be a measure space.
For all $s\in [1, \infty]$, the unit ball of $\mathscr{L}(L^s(S))$
is $\ell^{s}$-bounded, with constant $R^{s}(\mathscr{T})\leq 1.$
\end{example}

\begin{remark}
Let $p_i\in [1, \infty)$ and let $(S_i, \mu_i)$ be a measure space for $i=1, 2$.
Let $T:L^{p_1}(S_1)\to L^{p_1}(S_2)$ be a bounded operator. It follows from
\cite[Lemma 1.7]{Buh} that the singleton $\{T\}$ is $\ell^{1}$-bounded if and only if $T$ can be
written as the difference of two positive operators. In this result one can
replace $L^{p_i}(S_i)$ by certain Banach function spaces. This shows that
for an operator family $\mathscr{T}$ to be $\ell^{1}$-bounded imposes a rather
special structure on the operators in $\mathscr{T}$.
\end{remark}

Let $X$ be a Banach lattice. We denote by $X(\ell^s_N)$ the Banach space of all
sequences $(x_n)_{n=1}^N$ in $X$ endowed with
the norm
\[\|(x_n)_{n=1}^N\|_{X(\ell^s_N)} := \Big\|\Big(\sum_{n=1}^N |x_n|^s\Big)^{1/s}\Big\| ,\]
again with the obvious modification if $s=\infty$.
More details on these spaces can be found in \cite[p.\ 47]{LiTz}. Using this terminology,
the definition of $\ell^{s}$-boundedness can be rephrased as saying that
\begin{align}\label{eq:equivRs}\|(T_n x_n)_{n=1}^N\|_{Y(\ell_N^s)}
\leq C \|(x_n)_{n=1}^N\|_{X(\ell_N^s)}\end{align}
for all all finite sequences $(T_n)_{n=1}^N$ in $\mathscr{T}$ and
$(x_n)_{n=1}^N$ in $X$.

For $X = \R$ we have
$X(\ell^s_N)= \ell^s_N$ canonically for all $s\in [1,\infty]$.
For any Banach lattice $X$ the mapping
$$ (t\mapsto f_n(t))_{n=1}^N \ \mapsto \ \big(t\mapsto (f_n(t))_{n=1}^N\big)$$ establishes an
isometric isomorphism
\begin{align}\label{eq:Rs-LpX}
(L^p(S;X))(\ell^s_N) = L^p(S;X(\ell^s_N))
\end{align}
for all $p\in [1,\infty]$ and $s\in [1,\infty]$.

The following properties have been stated in \cite[Section 3.1]{UllmannPhD}. Recall that every
reflexive Banach lattice has
order continuous norm (see \cite[Section 2.4]{MN} for details).

\begin{proposition}\label{prop:suff}
Let $X$ and $Y$ be Banach lattices and let $s, s_0,s_1\in [1, \infty]$.
Let $\mathscr{T}\subseteq \calL(X,Y)$ be a family of bounded operators.
\begin{enumerate}
\item[\rm(1)] If $\mathscr{T}$ is $\ell^{s}$-bounded,
then also its strongly closed absolutely convex hull
$\overline{\rm{absco}}(\mathscr{T})$ is
$\ell^{s}$-bounded and
\[R^{s}\big(\overline{\rm{absco}}(\mathscr{T})\big) = R^{s}(\mathscr{T}).\]

\item[\rm(2)] The family
$\mathscr{T}$ is $\ell^{s}$-bounded if and only if the adjoint family
$\mathscr{T}^*$
is $\ell^{s'}$-bounded, where $\frac1s+\frac1{s'}=1$, and in this case we have
\[R^{s'}(\mathscr{T}^*) = R^{s}(\mathscr{T}).\]

\item[\rm(3)]
Suppose that $\mathscr{T}\subseteq\calL(X,Y)$ is both $\ell^{s_0}$-bounded and
$\ell^{s_1}$-bounded. If at least one of the spaces $X$ or $Y$ has order continuous norm,
then $\mathscr{T}$ is $\ell^{s_{\theta}}$-bounded for all $\theta\in (0,1)$,
where  $\frac1{s_{\theta}}=\frac{1-\theta}{s_0} + \frac{\theta}{s_1}$,
 and
\[R^{s_\theta}(\mathscr{T})\leq (R^{s_0}(\mathscr{T}))^{1-\theta} (R^{s_1}(\mathscr{T}))^{\theta}.\]
\end{enumerate}
\end{proposition}

For the proof of (1) one can repeat the analogous argument for $R$-boundedness (see
\cite[Theorem 2.13]{KuWe}). Assertion (2) follows from the identification $X(\ell^s_N)^* =
X^*(\ell^{s'}_N)$ (see \cite[p. 47]{LiTz}).
Assertion (3) follows by complex interpolation (see \cite[pages 57--58]{UllmannPhD}).

\section{$\ell^{s}$-Boundedness of convolution operators}\label{sec:det}

If $X$ is a Banach lattice and $J\subseteq \R_+$ is a finite subset,
for $f\in L_{\rm loc}^1(\R^d;X)$ we may define
\[(\widetilde{M}_J f)(\xi) := \sup_{r\in J} \frac1{|B_\xi(r)|}\int_{B_\xi(r)}  |f(\eta)|\, d\eta,\quad \xi\in\R^d,\]
where the modulus and supremum are taken in the lattice sense of $X$.

\begin{definition}\label{def:HL}
We say that $X$ has the {\em Hardy-Littlewood property} (briefly, $X$ {\em is an HL space})
if for all $p\in (1,\infty)$ and $d\ge 1$, all finite subsets $J\subseteq \R_+$, and all $f\in
L^p(\R^d;X)$ we have $\widetilde{M}_Jf\in L^p(\R^d;X)$ and there is a finite constant
$C = C_{p,d,X}\ge 0$, independent of $J$ and $f$, such that
\[
\|\widetilde{M}_J f\|_{L^p(\R^d;X)}\leq C\|f\|_{L^p(\R^d;X)}, \quad f\in L^p(\R^d;X).
\]
In this situation we will say that $\wt M$ is {\em bounded on $L^p(\R^d;X)$}.
\end{definition}
In \cite{GMT2} it has been proved that the Hardy-Littlewood property for fixed $p\in (1, \infty)$
and $d\geq 1$ implies the corresponding property for all $p\in (1, \infty)$
and $d\geq 1$, that is, the property is independent of $p\in (1, \infty)$
and $d\geq 1$.

In order to be able to deal with lattice suprema indexed by infinite sets $J$
we need to introduce some terminology. A Banach lattice
$X$ is called {\em monotonically complete} if $\sup_{i\in I} x_i$
exists for every norm bounded increasing net $(x_i)_{i\in I}$
(see \cite[Definition 2.4.18]{MN}). Recall the following two facts \cite[Proposition 2.4.19]{MN}:
\begin{itemize}
\item Every dual Banach lattice is monotonically complete.
\item If $X$ is monotonically complete, then  it has the weak Fatou property, i.e.,
there exists an $r$ only depending on $X$ such that
\[\big\|\sup_{i\in I} x_i\big\| \leq r\sup_{i\in I} \|x_i\|. \]
\end{itemize}

If $X$ is a monotonically complete HL space, then the Hardy-Littlewood maximal function
\[(\widetilde{M} f)(\xi) := \sup_{r>0} \frac1{|B_\xi(r)|}\int_{B_\xi(r)}  |f(\eta)|\, d\eta,\quad \xi\in\R^d,\]
is well-defined and bounded on each $L^p(\R^d;X)$.

It is known (see \cite[Theorem 2.8]{GMT1}) that HL spaces are $p$-convex for some
$p\in (1,\infty)$, i.e., there is a constant $C$ such that
$$ \Big\n \Big(\sum_{n=1}^N |x_n|^p \Big)^{1/p}\Big\n \le C_p \Big(\sum_{n=1}^N
\n x_n\n^p \Big)^{1/p}$$
for all finite subsets $x_1,\dots,x_N$ in $X$.
It is easy to check that $X=L^\infty$ has the HL property. In
\cite[Proposition 2.4, Remark 2.9]{GMT1} it is shown that $\ell^1$ fails the HL property.

The following deep result is proved in \cite{Bourgain-ext} and \cite[Theorem 3]{RF}.

\begin{proposition}\label{prop:UMD}
For a Banach function space $X$ the following assertions are equivalent:
\begin{enumerate}
\item[\rm(1)] $X$ is a UMD space;
\item[\rm(2)] $X$ and $X\s$ are HL spaces.
\end{enumerate}
\end{proposition}

We will be interested in the $\ell^{s}$-boundedness of the family of convolution operators
whose kernels $k\in L^1(\R^d)$ satisfy the almost everywhere pointwise bound
\begin{equation}\label{eq:pointbound}
 |k*f| \le Mf
 \end{equation}
for all simple $f:\R^d\to \R$. Let us denote by $\wt{\mathscr{K}}$ the set of all such kernels.

\begin{lemma}\label{lem:L1est}
For every $k\in \wt{\mathscr{K}}$ one has $\|k\|_{L^1(\R^d)}\leq 1$.
\end{lemma}
\begin{proof}
From
$$ |k*f(x) - k*f(x')| = \Big|\int_\R [k(x-y) - k(x'-y)]f(y)\,dy\Big|
\le \n k(x-\cdot)- k(x'-\cdot)\n_1 \n f\n_\infty$$
and the $L^1$-continuity of translations it follows that
 $k*f$ is a continuous function for $k\in L^1$ and $f\in L^\infty(\R^d)$.
For all functions $f\in L^\infty(\R^d)$ with
$|f|\leq 1$ almost everywhere, it follows from the assumption on $k$ and the observation just made
that for all $x\in \R^d$,
\begin{equation}\label{eq:kvareps}
 |k*f(x)|
 \leq \wt Mf(x)
 \leq 1.
\end{equation}

Consider the functions $f_n(y) := \text{sign}(k(-y))$.
Then \eqref{eq:kvareps} implies that
\[\int_{[-n,n]^d} |k(-y)|\, dy = |k*f(0)| \leq 1.\]
Letting $n$ tend to infinity, we find that $\|k\|_{L^1(\R^d)}\le
1$.
\end{proof}

Below we present classes of examples of such kernels. In particular, if a kernel
$k$ is radially decreasing, then $k\in \wt{\mathscr{K}}$ if and only if
$\|k\|_{L^1(\R^d)}\leq 1$ (see Proposition \ref{prop:Rbddkernelcond} below).

The next proposition shows that in \eqref{eq:pointbound} we may replace the
range space $\R$ by an arbitrary
Banach lattice. Of course, this result is trivial in the case of Banach function
spaces, where the estimate holds in a pointwise sense.

\begin{proposition}\label{prop:replaceX}
 Let $k\in \wt{\mathscr{K}}$ and let $X$ be a monotonically complete Banach lattice. If $f:\R^d\to X$ is a simple function, then almost everywhere
$$ |k*f| \le \wt M|f|.$$
\end{proposition}
\begin{proof}
For all $0\le x\s\in E\s$ and $\xi\in \R^d$ we have (see \cite[Proposition 1.3.7, Lemma 1.4.4]{MN})
\begin{align*}
\lb |k*f(\xi)|,x\s\rb & = \sup_{|y\s|\le x\s}\lb k*f(\xi), y\s\rb
\\ & = \sup_{|y\s|\le x\s} k*\lb f,y\s\rb(\xi)
\\ & \le  \sup_{|y\s|\le x\s} (\wt M \lb f,y\s \rb)(\xi)
\\ & \le \sup_{|y\s|\le x\s} \lb  (\wt M f)(\xi),|y\s| \rb
= \lb  (\wt Mf)(\xi),x\s \rb
\end{align*}
and the result follows \cite[Proposition 1.4.2]{MN}.
\end{proof}

The next result is well known in the scalar-valued case (see \cite[Chapter 2]{Grafakos1}).
The concise proof presented here was kindly shown to us by Tuomas Hyt\"onen.

\begin{proposition}\label{prop:Rbddkernelcond}
Let $X$ be a monotonically complete Banach lattice.
If $k:\R^d\to \R$ is a measurable function satisfying
\begin{equation*}
  \int_{\R^d}{\rm ess\,sup}_{|\eta|\geq |\xi|}|k(\eta)|\,d \xi\leq 1,
\end{equation*}
then $\|k\|_{L^1(\R^d)}\leq 1$ and for all $f\in L^p(\R^d;X)$, $1\le p\le \infty$,
we have the pointwise estimate $$|k*f|\leq \wt Mf.$$
\end{proposition}

\begin{proof}
By Proposition \ref{prop:replaceX}  it suffices to consider the case $X = \R$.
Put $h(r):=\esssup_{|\xi|\geq r}|k(\xi)|$.
Then $h$ is non-increasing, right-continuous and vanishes at infinity; hence
\begin{equation*}
  h(r)=\int_{(r,\infty)}\,d\mu(t)
\end{equation*}
for a positive measure $\mu$ on $\R_+=(0,\infty)$. Thus
\begin{align*}
  |k*f(\xi)|
  &=\Big|\int_{\R^d}k(\eta)f(\xi-\eta)\,d\eta\Big|
  \leq \int_{\R^d} h(|\eta|)|f(\xi-\eta)|\,d\eta
\\  &=\int_{\R^d} \int_{(|\eta|,\infty)}|f(\xi-\eta)|\,d\mu(t)\,d \eta \\
  &=\int_{\R_+}\int_{B(0,t)}|f(\xi-\eta)|\,d \eta\,d\mu(t)
  \leq\int_{\R_+}|B(0,t)|\wt Mf(\xi)\,d\mu(t).
\end{align*}
Further, writing $S(0,r)$ for the sphere in $\R^d$ of radius $r$ centered at the origin and $|S(0,r)|_{d-1}$ for its
$(d-1)$-dimensional measure, it follows by using polar coordinates that
\begin{align*}
   \int_{\R_+}|B(0,t)|\,d\mu(t)
   &=\int_{\R_+}\int_0^t|S(0,r)|_{d-1}\,d r \,d\mu(t)  \\
   &=\int_0^{\infty}\int_{(r,\infty)} |S(0,r)|_{d-1}\,d\mu(t)\,d r \\
   &=\int_0^{\infty}h(r)|S(0,r)|_{d-1}\,d r=\int_{\R^d}\esssup_{|\eta|\geq|\xi|}|k(\eta)|\,d\xi \leq 1.
\end{align*}
Hence $|k*f(\xi)|\leq \wt M|f|(\xi)$.
\end{proof}

The next result shows that the above sufficient condition holds under a certain integrability
condition on the derivative:

 \begin{proposition}\label{prop:K1}
  If $k\in W^{1,1}_{{\rm loc}}(\R^d\setminus\{0\})$ satisfies
  $\lim_{|\xi|\to \infty}k(\xi) = 0$,
   then
  $$ \int_0^\infty{\rm ess\,sup}_{|\eta|\ge|\xi|}|k(\eta)|\,d\xi
  \ \lesssim_d \ \int_{0}^\infty \rho^d {\rm ess\,sup}_{\xi\in S}|\nabla k(\rho\xi)|\, d\rho,
  $$
  where $S = S(0,1)$ is the unit sphere in $\R^d$. In particular, if the right-hand side is finite
  and $X$ is a monotonically complete Banach lattice,
  then for all $f\in L^p(\R^d;X)$, $1\le p\le \infty$, almost everywhere
  we have the pointwise estimate $$|k*f|\lesssim_d \Big(\int_{0}^\infty \rho^d \sup_{\xi\in S}|\nabla k(\rho\xi)|\, d\rho\Big)\wt M|f|.$$
 \end{proposition}

\begin{proof} Using polar coordinates,
\begin{align*}
 \int_{\R^d}{\rm ess\,sup}_{|\eta|\ge |\xi|}|k(\eta)|\,d\xi
 & =  \int_{\R^d} {\rm ess\,sup}_{\eta\in S} \sup_{r\ge |\xi|}|k(r\eta)|\,d\xi
 \\ & \le  \int_{\R^d} {\rm ess\,sup}_{\eta\in S} \sup_{r\ge |\xi|}\int_r^\infty |\nabla k(\rho\eta)|\,d\rho  \,d\xi
\\ & =\int_{\R^d} {\rm ess\,sup}_{\eta\in S} \sup_{r\ge |\xi|}\int_0^\infty \one_{(0,\rho)}(r)|\nabla k(\rho\eta)|\,d\rho  \,d\xi
\\ & \le  \int_0^\infty\int_{\R^d} {\rm ess\,sup}_{\eta\in S} \sup_{r\ge |\xi|} \one_{(0,\rho)}(r)|\nabla k(\rho\eta)|\,d\xi\, d\rho
\\ & \eqsim_d \int_{0}^\infty \rho^d {\rm ess\,sup}_{\eta\in S}|\nabla k(\rho\eta)|\, d\rho.
\end{align*}
Therefore, the result follows from Proposition \ref{prop:Rbddkernelcond}.
\end{proof}

Recall the definition $$ \wt{\mathscr{K}} := \{k\in L^1(\R^d): \ |k*f|\le \wt M|f|
\hbox{ a.e. for all simple $f:\R^d\to\R$}\}.$$
For a kernel $k\in L^1(\R^d)$ we denote by $T_k$ the associated convolution operator $f\mapsto k*f$
on $L^p(\R^d;X)$.

If $X$ is a UMD Banach function space and $s\in (1, \infty)$, then $X(\ell^s)$ is a
UMD Banach function space again  (see \cite[p.\ 214]{RF}). This implies that the
family
$\{T_k: k\in \wt{\mathscr{K}}\}$ is $\ell^{s}$-bounded. Indeed, using
\eqref{eq:equivRs}, for all finite sequences $(k_n)_{n=1}^N$ in $\wt{\mathscr{K}}$
and $(f_n)_{n=1}^N$ in $X(\ell^s_N)$ we have
\begin{align*}
\|(k_n* f_n)_{n=1}^N\|_{L^p(\R^d;X(\ell^s_N))}
& \leq \|(\wt M f_n)_{n=1}^N\|_{L^p(\R^d;X(\ell^s_N))}
\\ & \leq C_{X,s}  \|(f_n)_{n=1}^N\|_{L^p(\R^d;X(\ell^s_N))},
\end{align*}
where we applied Proposition \ref{prop:UMD} to $X(\ell^s_N)$. A similar but
simpler argument give that this result extends to $s=\infty$.

For $s=1$ this argument does not work since the maximal function is not bounded on $\ell^1$.
Surprisingly, we can still obtain the following result for $s=1$, which is the
main result of this section.
\begin{theorem}\label{thm:Xlq}
Let $X$ be a Banach lattice, let $p\in (1,\infty)$, and consider the family of convolution operators
$\wt{\mathscr{T}} = \{T_k: \ k\in \wt{\mathscr{K}}\}$ on $L^p(\R^d;X)$.
\begin{enumerate}[\rm(1)]
 \item If $X\s$ is an HL lattice, then
$\wt{\mathscr{T}}$ is $\ell^{1}$-bounded on $L^p(\R^d;X)$.
 \item If $X$ is a UMD Banach function space, then $\wt{\mathscr{T}}$ is $\ell^{s}$-bounded on
$L^p(\R^d;X)$ for all $s\in [1,\infty]$.
\end{enumerate}
\end{theorem}

\begin{remark}
It is crucial that the case $s=1$ is included here, i.e., the set
$\mathscr{T}$ is $\ell^{1}$-bounded on each $L^p(\R;X)$.
This fact will be needed in the proof of
our main result about $R$-boundedness of stochastic convolution operators (Theorem \ref{thm:Rbddstoch} below).
\end{remark}

Before turning to the proof of the theorem we start with some preparations and motivating results.
The next proposition shows that in the case of Banach function spaces, in a certain sense $\ell^{s}$-boundedness
of operator families becomes more restrictive as $s$ decreases.
\begin{proposition}
Let $X$ be a Banach function space. Let $1\leq s<t<\infty$ and $p\in (1, \infty)$ and $q = p t/s$. Let $\mathscr{T}_+ = \{T_k:k\in \mathscr{K}, k\geq 0\}$. If $\mathscr{T}_+$ is $\ell^{s}$-bounded on $L^p(\R^d;X)$, then $\mathscr{T}_+$ is $\ell^t$-bounded on $L^q(\R^d;X)$.
\end{proposition}
\begin{proof}
Let $0\le k_1, \dots k_N\in \mathscr{K}$ be non-negative kernels
and let $f_1, \dots, f_N:\R^d\to X$ be simple functions. By Lemma \ref{lem:L1est} we have
$\|k_n\|_{L^1(\R^d)}\leq 1$, and hence Jensen's inequality implies
$|k_n*f_n|^t\leq |k_n*(|f_n|^{t/s})|^s$. Therefore,
\begin{align*}
\|(|k_n*f_n|)_{n=1}^N \|_{L^q(X(\ell^t))} & \leq \|(|k_n*(f_n|^{t/s}))_{n=1}^N \|_{L^p(X(\ell^s))}^{s/t} \\ & \leq C \|(f_n|^{t/s})_{n=1}^N \|_{L^p(X(\ell^s))}^{s/t}  = C\|(f_n)_{n=1}^N \|_{L^q(X(\ell^t))}.
\end{align*}
\end{proof}

The next proposition gives necessary and sufficient conditions for $\ell^{\infty}$-bounded\-ness in terms
of $L^p$-boundedness of the maximal function $\widetilde M$.

\begin{proposition}\label{prop:Xlinfty}
Let $X$ be a Banach lattice, let $p\in [1, \infty]$ and consider the family of convolution operators $\wt{\mathscr{T}} = \{T_k: \ k\in \wt{\mathscr{K}}\}$ on $L^p(\R^d;X)$.
\begin{enumerate}[\rm(1)]
\item If $\wt{\mathscr{T}}$ is $\ell^{\infty}$-bounded on $L^p(\R^d;X)$, then $\widetilde{M}$ is $L^p(\R^d;X)$-bounded;
\item If $\widetilde{M}$ is $L^p(\R^d;X)$-bounded and $X$ is monotonically complete, then
 $\wt{\mathscr{T}}$ is $\ell^{\infty}$-bounded on $L^p(\R^d;X)$.
\end{enumerate}
\end{proposition}
Although the proof below also works for $p=1$, the maximal function is of course not bounded on $L^1(\R^d;X)$ (see \cite{Grafakos1}). As a consequence we see that $\wt{\mathscr{T}}$ is not $\ell^{\infty}$-bounded on $L^1(\R;X)$.
\begin{proof}
(1): \
For all $r>0$ and simple $f:\R^d\to \R$
we have
\begin{align*} \frac1{|B_0(r)|}|\one_{B_0(r)}* f(\xi)|
& = \frac1{|B_0(r)|}\Big|\int_{\R^d} \one_{B_0(r)}(\xi-\eta)f(\eta)\,d\eta\Big|
\\ & = \frac1{|B_\xi(r)|} \Big|\int_{\R^d} \one_{B_\xi(r)}(\eta)f(\eta)\,d\eta\Big|.
\end{align*}
It follows that the functions $k_r := \frac1{|B_0(r)|}\one_{B_0(r)} $ belong to
$\wt{\mathscr{K}}$ for all $r>0$. Moreover, the above identities extend to functions
$f\in L^p(\R^d;X)$ provided we interpret $|\cdot|$ as the modulus in $X$.
As a consequence, for all $f\in L^p(\R^d;X)$ and all finite sets $J\subseteq\R_+$,
the $\ell^{\infty}$-boundedness of  $\mathscr{T}$ on $L^p(\R^d;X)$ implies
\begin{align*}
\n \widetilde{M}_Jf\n_{L^p(\R^d;X)}
 = \Big\n \sup_{r\in J} |T_{k_r} f|\Big\n_{L^p(\R^d;X)}
 \lesssim_{p,d,X} \n f\n_{L^p(\R^d;X)}.
\end{align*}
It follows that the mappings $\widetilde{M}_J$ are bounded on $L^p(\R^d;X)$,
uniformly with respect to $J$.
\smallskip

(2): \
Let $k_1, \ldots, k_N\in \wt{\mathscr{K}}$ and simple $f_1,\dots,f_N\in L^p(\R^d;X)$ be given.
Then, by Proposition \ref{prop:replaceX},
\begin{align*}
\int_{\R^d} \Big\| \sup_{1\leq n\leq N} |T_{k_n} f_n(\xi)| \Big\| ^p \, d\xi & \leq
\int_{\R^d} \Big\| \sup_{1\leq n\leq N} (\widetilde{M} |f_n|)(\xi)| \Big\| ^p \, d\xi
\\ & \leq \int_{\R^d} \Big\| \widetilde{M} (\sup_{1\leq n\leq N} |f_n|)(\xi) \Big\| ^p \, d\xi
\\ & \lesssim_{p,d,X}\int_{\R^d} \Big\| \sup_{1\leq n\leq N} |f_n(\xi)| \Big\| ^p \, d\xi,
\end{align*}
with the obvious modifications for $p=\infty$. By approximation, this estimate extends to
general $f_1,\dots,f_N\in L^p(\R^d;X)$.
\end{proof}

\begin{proof}[Proof of Theorem \ref{thm:Xlq}]
Fix $1<p<\infty$.
We begin by observing that for $k\in L^1(\R^d)$ the adjoint of $T_k$ as an operator on $L^p(\R^d)$
equals $T_{\bar{k}}$ as an operator on $L^{p'}(\R^d)$, $\frac1p+\frac1{p'}=1$ and $\bar{k}(x) = k(-x)$.
Clearly, $T_{\bar{k}}\in \mathscr{T}$.

By definition (if $X\s$ is HL), respectively by Proposition \ref{prop:UMD}
(if $X$ is UMD), $\widetilde{M}$ is
$L^{p'}(\R^d;X^*)$-bounded.
Therefore, by Proposition \ref{prop:Xlinfty},
$\mathscr{T}$
is $\ell^{\infty}$-bounded on $L^{p'}(\R^d;X^*)$, and Proposition \ref{prop:suff}(2) then
shows that $\mathscr{T}$
is $\ell^{1}$-bounded on $L^{p}(\R^d;X)$.

If $X$ is UMD, we have already sketched a proof in the case $s\in (1, \infty]$ (
alternatively we can use interpolation). We may apply the above argument to $p'$ and $X\s$
as well and obtain that $\mathscr{T}$
is also $\ell^{\infty}$-bounded on $L^{p}(\R^d;X)$. Now the result follows from
Proposition \ref{prop:suff}(3). Here we used that a UMD space $X$ is reflexive and thus
$L^p(\R^d;X)$ is reflexive (see \cite{DU}) and hence has order continuous norm.
\end{proof}

\begin{remark}
If a family of kernels $\mathscr{K}=\{k:\R^d\to \R\}$ satisfies an appropriate smoothness condition,
then the $\ell^{s}$-boundedness of $\{T_k:k\in \mathscr{K}\}$ as a family of operators on
$L^{p_0}(\R^d;X)$ for a certain $p_0\in [1, \infty]$ implies the $\ell^{s}$-boundedness on $L^{p}(\R^d;X)$ for
all $p\in(1, \infty)$ (see \cite[Theorem V.3.4]{GCRdF}). This result is interesting
from a theoretical point of view, but in all applications considered here we can consider
arbitrary $p\in (1, \infty)$ from the beginning without additional difficulty. The main
reason for this is the $p$-independence of the HL property.
\end{remark}

The next example shows that Theorem \ref{thm:Xlq} does not extend to $p=1$.

\begin{example}\label{ex:nontriviallity}
Let $X = \ell^r$ with $r\in (1, \infty)$ fixed. By Theorem \ref{thm:Xlq},
the family $\mathscr{T}_1$ considered there is $\ell^{s}$-bounded
on $L^p(\R;\ell^r)$ for all $p\in(1,\infty)$ and $s\in [1, \infty]$.
We show that it fails to be $\ell^{s}$-bounded on $L^1(\R;\ell^r)$ for all $s\in [1,\infty]$.

Let $\lambda_n>0$ with $\lambda_n\to \infty$ as $n\to \infty$. Let $k_n(t)  =
\frac12\lambda_n e^{-\lambda_n |t|}$. By Proposition \ref{prop:Rbddkernelcond} we have $(k_n)_{n\ge 1}\subseteq \wt{\mathscr{K}}$
and$\{T_{k_n}: n\geq 1 \}\subseteq \wt{\mathscr{T}}$. The kernels $(k_n^2)_{n\geq 1}$, are precisely the ones which are needed in \cite[Section 7]{NVW13}.

Fix $s\in [1,\infty]$. We will
show that $\{T_{k_n}: n\geq 1 \}$ is not
$\ell^{s}$-bounded as a family of operators on $L^1(\R;\ell^r)$.
Indeed, assume it is $\ell^{s}$-bounded on this space with constant $C$. Then,
letting $N\to \infty$ in the definition of $\ell^{s}$-boundedness and using
the identification $(L^1(\R;\ell^r))(\ell^s) = L^1(\R;\ell^r(\ell^s))$
(see \eqref{eq:Rs-LpX}), we obtain
\begin{align*}
\int_{\R} \Big(\sum_{j\geq 1} \Big[\Big(\sum_{n\geq 1} \Big|\int_{\R} k_n(u)
f_{nj}(t-u)\, &du\Big|^s \Big)^{1/s}\Big]^{r}\Big)^{1/r} dt
\\ & \leq
C\int_{\R} \Big(\sum_{j\geq 1} \Big[\Big(\sum_{n\geq 1}
|f_{nj}(t)|^s\Big)^{1/s} \Big]^{r}\Big)^{1/r} dt,
\end{align*}
where $(f_{nj})_{n,j\geq 1}$ is in $L^1(\R;\ell^r(\ell^s))$; we make the obvious modifications
if $s=\infty$.
Taking $f_{nj}(t) = f_j(t) \delta_{nj}$ yields
\begin{align*}
\int_{\R} \Big(\sum_{j\geq 1} \Big|\int_{\R} k_j(u) f_{j}(t-u)\, du\Big|^{r}
\Big)^{1/r} \, dt \leq C\int_{\R} \Big(\sum_{j\geq 1}
|f_{j}(t)|^{r}\Big)^{1/r} dt.
\end{align*}
The latter is easily seen
to be equivalent to the maximal $L^1$-regularity of the diagonal
operator $A e_j = \lambda_j e_j$ on $\ell^r$, which does not hold by
\cite{Guerre}.
\end{example}

\section{The operators $N_k$}

Let $X$ be a Banach space and $H$ be a Hilbert space.
For $k\in L^2(\R^d)$ and simple functions $G:\R^d\to H\otimes X$ we define
the function $N_k G: \R^d\to L^2(\R^d)\otimes H\otimes  X$ by
\begin{equation}\label{eq:Rkdef}
((N_k G)(t))(s) = k(t-s) G(s), \ \ s,t\in \R^d.
\end{equation}

\begin{proposition}\label{prop:Rk-type2}
Let $X$ be a non-zero Banach space, $H$ a non-zero Hilbert space,
and let  $p\in [1,\infty)$ and $d\ge 1$ be arbitrary and fixed.
The following assertions are equivalent:

\begin{enumerate}[\rm(1)]
\item  $X$ has type $2$ and $p\in [2,\infty)$;
\item For all $k\in L^2(\R^d)$ the operator $G\mapsto N_kG$ extends to a bounded operator
$$N_k: L^p(\R^d;\gamma(H,X)) \to L^p(\R^d,dt;\g(L^2(\R^d;H,ds),X)).$$
\end{enumerate}
\end{proposition}

\begin{proof}
(1)$\Rightarrow$(2):
It suffices to prove that
for any Banach space $Y$ with type $2$
the mapping $G\mapsto N_kG$ extends
to a bounded operator  $$N_k: L^p(\R^d;Y)\to L^p(\R^d,dt;L^2(\R^d,ds;Y)).$$
Indeed, once this has been shows we take $Y = \gamma(H,X)$ (which has type $2$
if $X$ has type $2$) and apply Proposition \ref{prop:L2gamma}(1).

Fix $p\ge 2$ and $f\in L^p(\R^d;Y)$. By Young's inequality,
\begin{align*}
\n N_k f\n_{L^p(\R^d,dt;L^2(\R^d,ds;Y))}^p
 &  = \int_{\R^d} \Big(\int_{\R^d} |k(t-s)|^2 \n f(s)\n ^2 \, ds\Big)^{p/2} \, dt
\\ & = \big\n |k|^2 * \n f\n ^2 \big\n_{L^{p/2}(\R^d)}^{p/2}
\\ & \le \|k^2\|_{L^1(\R^d)}^{p/2} \big\n  \n f\n ^2 \big\n_{L^{p/2}(\R^d)}^{p/2}
= \|k\|_{L^2(\R^d)}^{p}  \n f\n_{L^{p}(\R^d;Y)}^{p}.
\end{align*}

(2)$\Rightarrow$(1): To show that $p\in [2, \infty)$ it suffices to argue on one-dimensional
subspaces of $X$. We may therefore assume that $X=H=\R$ and therefore $\g(L^2(\R^d);X) = L^2(\R^d)$.

Let $k = \one_{(a,b)^d}$ with $a<b$ and set $\delta := b-a$.
For $0<r<\delta/2$ let $G_r := \one_{(0,r)^d}$. Then
\begin{align*}
\|N_k G_r\|_{L^p(\R^d,dt;L^2(\R^d,ds))}^p & = \int_{\R^d} \Big(\int_{\R^d} k^2(s) G_r^2(t-s) \, ds\Big)^{p/2} \, dt
\\ & \geq \int_{((a+b)/2,b)^d} \Big(\int_{t-(0, r)^d} 1 \, ds\Big)^{p/2} \, dt = r^{dp/2}(\delta/2)^d.
\end{align*}
On the other hand, $\|G_r\|_{L^p(\R^d)}^p = r^d$.
Therefore, $\|N_k\|\geq  r^{d(\frac12-\frac1p)}(\delta/2)^d$.
Letting $r\downarrow 0$, the boundedness of $N_k$ forces that $p\in [2, \infty)$.

To show that $X$ has type $2$ we may assume that $H=\R$ (identify $X$ with a
closed subspace of $\gamma(H,X)$ via the mapping $x\mapsto h_0\otimes x$, where
$h_0\in H$ is some fixed norm one vector).

As before let $k = \one_{(a,b)^d}$ with $a<b$ and fix $0<r\le \delta/2$ with $\delta = b-a$.
Fix a simple function $G:\R^d\to X$ with
support in $I = (0,r)^d$. For all $t\in J= ((a+b)/2, b)^d$,
one has
\[\| G\|_{\g(L^2(I),X)} = \|k(t-\cdot) G\|_{\g(L^2(I),X)}.\]
It follows that
\begin{align*}
|J|^{1/p} \| G\|_{\g(L^2(I),X)} & = \|t\mapsto k(t-\cdot) G\|_{L^p(J;\g(L^2(I),X))}
\\ & \leq \|t\mapsto k(t-\cdot) G\|_{L^p(\R^d;\g(L^2(\R^d),X))}
\\ & = \|N_k G\|_{L^p(\R^d;\g(L^2(\R^d),X))} \\ & \leq \|N_k\| \, \|G\|_{L^p(\R^d;X)} = \|N_k\| \,
\|G\|_{L^p(I;X)}.
\end{align*}
As a consequence, the identity mapping on $L^p(I)\otimes X$
extends to a bounded operator from
$L^p(I;X)$ to $\g(L^2(I),X).$
Hence by \cite[Proposition 6.1]{RS}, $X$ has type $2$.
\end{proof}
Inspection of the proof shows that the following weaker version of (2) already implies
(1):
\medskip

\begin{enumerate}[\rm(1)]
 \item[\rm(2$'$)] {\em There exist real numbers $a<b$ such that
 the mapping $G\mapsto N_{\one_{(a,b)^d}}G$
 extends to a bounded operator}
$$N_{\one_{(a,b)^d}}: L^p(\R^d;X) \to L^p(\R^d;\g(L^2(\R^d),X))).$$
\end{enumerate}
In view of this we shall assume from now on that $X$ has type $2$ and consider only exponents
$p\in [2,\infty)$.
We now fix a subset $\mathscr{K}\subseteq L^2(\R^d)$
and consider the family $$\mathscr{N}_{\mathscr{K}} := \{N_k: k\in \mathscr{K}\}.$$
By the previous result,
the operators in $\mathscr{N}_{\mathscr{K}}$ extend to bounded operators
from $L^p(\R^d;X)$ to
$L^p(\R^d;\g(L^2(\R^d),X))$. By slight abuse of notation, the resulting family
of extensions will be denoted by $\mathscr{N}_{\mathscr{K}}$ again.

In the next result we investigate the role of $H$ with regard to the $R$-boundedness properties
of $\mathscr{N}_{\mathscr{K}}$.

\begin{proposition}[Independence of $H$]\label{prop:alphaplus}
Let $X$ be a Banach space with type $2$, $H$ be a non-zero Hilbert space,
and $p\in [2,\infty)$. For any set
$\mathscr{K}\subseteq L^2(\R^d)$, the following assertions are equivalent:
\begin{enumerate}[\rm(1)]
\item The family $\mathscr{N}_{\mathscr{K}}$
\!is $R$-bounded from $L^p(\R^d;X)$ to $
L^p(\R^d;\g(L^2(\R^d),X))$;
\item The family
$\mathscr{N}_{\mathscr{K}}$
\!is $R$-bounded from $L^p(\R^d;\g(H,X))$ to
$L^p(\R^d;\g(L^2(\R^d;H),X))$.
\end{enumerate}
\end{proposition}
\begin{proof}
We only need to prove that
(1) implies (2); the converse implication follows by restricting to a one-dimensional subspace of $H$
and identifying $\gamma(\R,X)$ with $X$.

Suppose now that (1) holds. By Proposition \ref{prop:Rbdd-ext} each operator
in $\mathscr{N}_{\mathscr{K}}$ extends to a bounded operator
from $\gamma(H,L^p(\R^d;X))$ to $\gamma(H, L^p(\R^d;\g(L^2(\R^d),X)))$ and the resulting
family of extensions is again $R$-bounded.
By the $\gamma$-Fubini isomorphism (Proposition \ref{prop:g-Fub}),
$\mathscr{N}_{\mathscr{K}}$
extends to an $R$-bounded family of operators from
$L^p(\R^d;\gamma(H,X))$ to $L^p(\R^d;\gamma(H,\g(L^2(\R^d),X)))$. Now the result follows from the fact that
$\gamma(H,\g(L^2(\R^d),X))$ embeds continuously into $\g(L^2(\R^d;H),X)$ by
Propositions \ref{prop:alpha} and \ref{prop:typealpha}.
\end{proof}

The main result of this section reduces the problem of proving $R$-boundedness of a family of operators $N_k$
to proving $\ell^{1}$-boundedness of the corresponding family of
convolution operators $T_{k^2}$ (see Section \ref{sec:det} for the definition of these operators).

We recall from Proposition \ref{prop:type-Mtype} and its proof that a Banach lattice has type $2$ if and only if it has martingale
type $2$, and that such a Banach lattice is $2$-convex. Because of this, its $2$-concavification
$X^2$ is a Banach lattice again. If $X$ is a Banach function space
over some measure space $(S,\mu)$ (this is the only case we shall consider),
$X^2$ consists of all measurable functions $f: S\to \R$ such that $|f| = g^2$ for
some $g\in X$,
identifying functions which are equal $\mu$-almost everywhere. For example, when $X = L^q(S)$ with $q\in [2,\infty)$,
then $X^2 = L^{q/2}(S)$.

\begin{theorem}\label{thm:workinglemma}
Let $X$ be a Banach lattice with type $2$ and let $X^2$
denote its $2$-concavification. Let $p\in [2,\infty)$.
For any set of kernels $\mathscr{K}\subseteq  L^2(\R^d)$, the following assertions are equivalent:
\begin{enumerate}[\rm(1)]
\item
The family $$\mathscr{N}_\mathscr{K} := \{N_k: k\in \mathscr{K}\}$$ is $R$-bounded
from $L^p(\R^d;X)$ to $L^p(\R^d;\g(L^2(\R^d),X))$;
\item The family $$\mathscr{T}_{\mathscr{K}^2} : = \{T_{k^2}: k\in \mathscr{K}\}$$ is $\ell^{1}$-bounded
 on $L^{p/2}(\R^d;X^2)$.
\end{enumerate}
Moreover, $R(\mathscr{N}_\mathscr{K})\eqsim_{p,X} (R^{1}(\mathscr{T}_{\mathscr{K}^2}))^{1/2}$.
\end{theorem}

\begin{proof}

(2) $\Rightarrow$ (1): \  Assume that $\mathscr{T}_{\mathscr{K}^2}$ is $\ell^{1}$-bounded and fix $k_1,
\ldots, k_N\in \mathscr{K}$. Let $G_1,\dots, G_N\in L^p(\R^d;X)$ be simple functions.
As $X$ has type $2$, it also has finite cotype
(see \cite[Example 11.1.2 and Theorem 11.1.14]{AlKa}).
Since each $G_n$ takes values in a finite-dimensional subspace of $X$, a standard argument shows
that we may assume $X$ is a Banach function space (see \cite[Theorem 3.9]{VerEmb}).

Set $f_n: = |G_n|^2$. By Lemma \ref{lem:squarefunction},
\begin{align*}
 \E_r \Big\|\sum_{n=1}^N r_n (N_{k_n} G_n)(t) \Big\|_{\g(L^2(\R^d),X)}^p
\! & = \E_r
\Big\|s\mapsto \sum_{n=1}^N r_n k_n(t-s) G_n(s)
 \Big\|^p_{\g(L^2(\R^d),X)}
\\ & \eqsim_{p,X} \Big\|  \Big(\sum_{n=1}^N  \int_{\R^d} |k_n(t-s)
G_n(s)|^2  \, ds\Big)^{1/2} \Big\| ^p
\\ & = \Big\| \Big(\sum_{n=1}^N   \int_{\R^d} k_n^2(t-s) f_n(s) \, ds  \Big)^{1/2}
\Big\| ^p
\\ & = \Big\| \Big(\sum_{n=1}^N  k_n^2* f_n(t)  \Big)^{1/2} \Big\| ^p
\\ & = \Big\|\sum_{n=1}^N  k_n^2* f_n(t)   \Big\|_{{X^2}}^{p/2}.
\end{align*}

Integrating over $\R^d$, it follows that
\begin{align*}
\E_r \Big\|\sum_{n=1}^N r_n (N_{k_n} G_n) \Big\|_{L^p(\R^d;\g(L^2(\R^d),X))}^p
& =
\int_{\R^d} \Big\|\sum_{n=1}^N  k_n^2* f_n \Big\|_{{X^2}}^{p/2} \, dt
\\ & =
\Big\|\sum_{n=1}^N  k_n^2* f_n \Big\|_{L^{p/2}(\R^d;{X^2})}^{p/2}
\\ & \leq
(R^{1}(\mathscr{T}_{\mathscr{K}^2}))^{p/2} \Big\|\sum_{n=1}^N f_n\Big\|_{L^{p/2}(\R^d;{X^2})}^{p/2}.
\end{align*}
Now for the latter one has
\begin{align*}
\Big\|\sum_{n=1}^N f_n\Big\|_{L^{p/2}(\R^d;{X^2})}^{p/2} & = \Big\|\Big(\sum_{n=1}^N
|G_n|^2\Big)^{1/2}\Big\|_{L^{p}(\R^d;X)}^{p}
\\ & \eqsim_X \Big\| \sum_{n=1}^N r_n G_n \Big\|_{L^{p}(\R^d;L^2(\O_r;X))}^{p}
\leq \E_r \Big\| \sum_{n=1}^N r_n G_n \Big\|_{L^{p}(\R^d;X)}^{p}.
\end{align*}
Combing the estimates, the result follows.

(1) $\Rightarrow$ (2): \ This is proved similarly.
\end{proof}

\section{Stochastic integration}\label{sec:stoch}

We begin recalling some basic facts from the theory of stochastic integration in
UMD Banach spaces as developed in \cite{NVW1} (for a survey see \cite{NVW13}).

Let $(\O,\P)$ be a probability space and let $H$ be a Hilbert space.
An {\em $H$-cylindrical Brownian motion}
is a bounded linear operator $W_H$ from $L^2(\R_+;H)$ to $L^2(\O)$ such
that
\begin{enumerate}
\item[\rm(i)] for all $f\in L^2(\R_+;H)$ the random variable $W_H f$ is centered Gaussian;
\item[\rm(ii)]  for all $f,g\in L^2(\R_+;H)$ we have
$ \E (W_H f \cdot W_H g) = [f,g]_{L^2(\R_+;H)}.$
\end{enumerate}
If $(\O,\P)$ is endowed with a filtration $\F = (\F_t)_{t\in \R_+}$, we call
$W_H$ a {\em $H$-cylindrical $\F$-Brownian motion} on $H$ if $W_H f$ is
independent of $\F_t$ for all $f\in L^2(\R_+;H)$ with support in $(t,\infty)$.
In that case, $t\mapsto W_H (\one_{(0,t)}\otimes h)$ is an $\F$-Brownian motion
for all $h\in H$, which is standard if $\n h\n=1$. Two such Brownian motions
are independent if and only if the corresponding vectors $h$ are orthogonal.
If there is no danger of confusion we also use the standard notation $W_H(t)
h$ for the random variable $W_H(\one_{(0,t)}\otimes h)$.

For $0\le a<b<\infty$, $x\in X$, and an $\F_a$-measurable set $A\subseteq  \O$, the
stochastic integral of the indicator process $(t,\omega)\mapsto \one_{(a,b]\times
A}(t, \omega)\,h\otimes x$ with respect to $W_H$ is defined as
\[\int_{0}^t \one_{(a,b]\times A}\otimes (h\otimes x)\,dW_H
:= \one_A \,W_H(\one_{(a\wedge t, b\wedge t]}\otimes h) \otimes x, \quad t\in \R_+.\]
By linearity, this definition extends to {\em adapted finite rank step
processes}, which we define as finite linear combinations of  indicator
processes of the above form.

\begin{proposition}[Burkholder inequality for martingale type $2$ spaces; see \cite{Brz1, Brz2, Ondrejat04}]\label{prop:Burkholder}
Let $X$ have martingale type $2$ and let $p\in (1, \infty)$ be fixed.
For all
adapted finite rank step processes $G$ we have
$$ \E \sup_{t\in\R_+}\Big\n \int_0^t G\,dW_H\Big\n^p \le C_{p,X}^p \|G\|_{L^p(\O;L^2(\R_+;\gamma(H,X)))}^p.$$
\end{proposition}

Under the identification
$$
\one_{(a,b]\times
A} \otimes (h\otimes x) = \one_A \otimes ((\one_{(a,b]}\otimes h)\otimes x),
$$ we may identify finite rank step processes
with elements in $L^p(\O;\g(L^2(\R_+;H),X)))$ and we
have the following estimate.

\begin{proposition}\cite[Theorems 5.9, 5.12]{NVW1}\label{prop:NVW}
Let $X$ be a UMD Banach space and let $p\in (1,\infty)$ be fixed. For all
adapted finite rank step processes $G$ we have
$$
c^p\E\n G\n_{\g(L^2(\R_+;H),X)}^p \le \E \sup_{t\in \R_+}\Big\n \int_0^t G\,dW_H\Big\n^p
\le C^p \E\n G\n_{\g(L^2(\R_+;H),X)}^p,
$$
with constants $0<c\le C<\infty$ independent of $G$.
\end{proposition}
When $G$ does not depend on $\O$ the UMD condition can be omitted in the above result.

By a standard density argument (see \cite{NVW1} for details),
the stochastic integral has a unique extension to the Banach space
$L_\F^p(\O;\g(L^2(\R_+;H),X))$ of all adapted elements in
$L^p(\O;\g(L^2(\R_+;H),X))$, that is, the closure in $L^p(\O;\g(L^2(\R_+;H),X))$
of all adapted simple processes
with values in $H\otimes X$.
 In the remainder of this paper, all stochastic
integrals are understood in this sense.

\subsection{Stochastic convolution operators}

For kernels $k\in L^2(\R_+)$ and adapted finite rank step processes
$G:\R_+\times\O\to H\otimes X$ we define the adapted process $S_{k} G:\R_+\times\O\to
X$ by
\begin{align}\label{eq:Ikoperator}
S_{k}^H G(t):= \int_0^t k(t-s) G(s) \, dW_H(s), \ \ t\in \R_+.
\end{align}
Since $G$ is an adapted finite rank step process,
the It\^o integration theory for scalar-valued processes  (see \cite[Chapter 17]{Kal}) shows
that the above stochastic integral is well defined for all $t\in \R_+$.

The following observation is a direct consequence of Proposition \ref{prop:Burkholder} and Young's inequality.
\begin{proposition} Let $X$ be a Banach space, $H$ a Hilbert space, and $p\in [2,\infty)$.
If $X$ has martingale type $2$, the mapping $S_k^H: G\mapsto S_k^HG$ extends to a bounded operator
from
$L^p_\F(\R_+\times \Omega;\g(H,X))$ to $L^p(\R_+\times \O;X)$.
\end{proposition}

Note that for deterministic integrands
\[\|S_k^H G(t)\|_{L^p(\O;X)} \eqsim_p \|s\mapsto k(t-s) G(s)\|_{\gamma(L^2(0,t;H),X)}.\]
Therefore, from the proof of Proposition \ref{prop:Rk-type2} we can deduce the following result:

\begin{proposition} Let $X$ be a Banach space, $H$ a non-zero Hilbert space, and $p\in [2,\infty)$.
The following assertions are equivalent:
\begin{enumerate}[\rm(1)]
\item  $X$ has type $2$;
\item For all $k\in L^2(\R_+)$ the mapping $S_k: G\mapsto S_kG$ extends to a bounded operator
from $L^p_\F(\R_+;X)$ into $L^p(\R_+\times\Omega;X)$;
\item For all $k\in L^2(\R_+)$ the mapping  $S_k^H: G\mapsto S_k^HG$ extends to a bounded operator
from $L^p_\F(\R_+;\g(H,X))$ into $L^p(\R_+\times\O;X)$.
\end{enumerate}
\end{proposition}

\section{$R$-boundedness of stochastic convolution operators}

We shall now apply the results of
Section \ref{sec:det}
to obtain $R$-boundedness results for stochastic convolution operators.
More specifically, we shall provide a connection between $R$-boundedness of stochastic convolutions
with kernel $k$ and $\ell^{1}$-boundedness of convolutions with the squared kernel $k^2$. For $d=1$, the results of the previous section imply their counterparts for $\R_+$ by considering functions and kernels supported on $\R_+$.

Recall that for $k\in L^2(\R_+)$ the stochastic convolution
operators $S_k$ have been defined by \eqref{eq:Ikoperator}.
For a subset $\mathscr{K}\subseteq L^2(\R_+)$ we write
$\mathscr{S}_{\mathscr{K}} := \{S_{k}: k\in \mathscr{K}\}$; we use the same notation
for the vector-valued extensions. We will be interested in the $R$-boundedness
of such families.
The first result asserts that it suffices to check $R$-boundedness on deterministic integrands:

\begin{theorem} \label{thm:main1}
Let $X$ be a Banach space with type $2$, $H$ be a non-zero
Hilbert space, and let $p\in [2,\infty)$.
For a set $\mathscr{K}\subseteq L^2(\R_+)$ the following assertions are equivalent:
\begin{enumerate}[\rm(1)]
 \item The family $\mathscr{S}_{\mathscr{K}}$
is $R$-bounded from
$L^{p}(\R_+;X)$ to $L^{p}(\R_+\times\Omega;X)$;
 \item The family $\mathscr{S}_{\mathscr{K}}^H$
is $R$-bounded from
$L^{p}(\R_+;\g(H,X))$ to $L^{p}(\R_+\times\Omega;X)$;
\item The family $\mathscr{N}_{\mathscr{K}}$
is $R$-bounded from $L^p(\R_+;X)$ to $L^p(\R_+;\g(L^2(\R_+),X))$.
\end{enumerate}
If $X$ has martingale
type $2$, the assertions {\rm(1)}--{\rm(3)} are equivalent to
\begin{enumerate}[\rm(1)]
\item[\rm(4)] The family $\mathscr{S}_{\mathscr{K}}$
is $R$-bounded from
$L^{p}_\mathscr{F}(\R_+\times\O;X)$ to $L^{p}(\R_+\times\O;X)$;
\item[\rm(5)] The family $\mathscr{S}_{\mathscr{K}}^H$
is $R$-bounded from
$L^{p}_\mathscr{F}(\R_+\times\O;\g(H,X))$ to $L^{p}(\R_+\times\O;X)$;
\end{enumerate}
If, moreover, $X$ is a Banach lattice (in which case the type $2$ assumption and the martingale
type $2$ assumption are equivalent), the assertions {\rm(1)}--{\rm(5)} are equivalent to
\begin{enumerate}[\rm(1)]
\item[\rm(6)] The family $\mathscr{T}_{\mathscr{K}^2} : = \{T_{k^2}: k\in \mathscr{K}\}$ is $\ell^{1}$-bounded
 on $L^{p/2}(\R_+;X^2)$.
\end{enumerate}
In all equivalences, the $R$-bounds are comparable with constants depending only on $p$ and $X$.
\end{theorem}
\begin{proof}
The implications (2) $\Rightarrow$ (1), (4) $\Rightarrow$ (1), (5) $\Rightarrow$ (2),
(5) $\Rightarrow$ (4) are trivial, and for Banach lattices $X$ the equivalence
(3) $\Leftrightarrow$ (6) is the content of Theorem \ref{thm:workinglemma}.

(2) $\Rightarrow$ (5): Assuming that (2) holds,
for any choice of $S_{k_1},\dots, S_{k_N}\in \mathscr{S}_{\mathscr{K}}^H$ and
$G_1,\dots, G_N \in L^{p}_\mathscr{F}(\R_+\times\O;\g(H,X))$ we have, by Fubini's theorem
and (1),
\begin{align*}
 \E_r \Big\n \sum_{n=1}^N r_n S_{k_n} G_n \Big\n_{L^{p}(\R_+\times\O;X)}^2
& = \E \E_r \Big\n \sum_{n=1}^N r_n S_{k_n} G_n \Big\n_{L^{p}(\R_+;X)}^2
\\ & \le \rho^2 \E \E_r \Big\n \sum_{n=1}^N r_n G_n \Big\n_{L^{p}(\R_+;\gamma(H,X))}^2
\\ & = \rho^2 \E_r \Big\n \sum_{n=1}^N r_n G_n \Big\n_{L^{p}(\R_+\times\O;\gamma(H,X))}^2,
\end{align*}
with $\rho$ the $R$-boundedness constant as meant in (2).

(1) $\Leftrightarrow$ (3):
Fix $k_1, \ldots, k_N\in \mathscr{K}$ and let $G_1, \ldots, G_N$ be
elements of $L^p(\R_+;X)$.
 By Proposition \ref{prop:NVW}, for all $t\in \R_+$ we have
\begin{align*}
\E_r \Big\|\sum_{n=1}^N r_n (S_{k_n}^H G_n)(t)\Big\|_{L^p(\O;X)}^p &= \E_r \Big\|
\int_{\R_+} \sum_{n=1}^Nr_n k_n(t-s) G_n(s) \, d W(s) \Big\|_{L^p(\O;X)}^p
\\ & \eqsim_{p,X} \E_r \Big\|s\mapsto \sum_{n=1}^N r_n k_n(t-s)
G_n(s)\Big\|_{\g(L^2(\R_+),X)}^p
\\ & = \E_r \Big\|\sum_{n=1}^N r_n (N_{k_n} G_n)(t) \Big\|_{\g(L^2(\R_+),X)}^p.
\end{align*}
An integration over $t$ gives
\[\E_r \Big\|\sum_{n=1}^N r_n S_{k_n}^H G_n\Big\|_{L^p(\R_+\times\O;X)}^p
\eqsim_{p,X} \E_r \Big\|\sum_{n=1}^N r_n N_{k_n} G_n \Big\|_{L^p(\R_+;\g(L^2(\R_+),X))}^p. \]

(2) $\Leftrightarrow$ (3): The same argument as in the proof of (1) $\Leftrightarrow$ (3)
can be shown that (2) is equivalent with (3)$'$, where
\begin{enumerate}[\rm(1)]
\item[\rm(3)$'$] $\mathscr{N}_{\mathscr{K}}$
is $R$-bounded from $L^p(\R_+;\gamma(H,X))$ to $L^p(\R_+;\g(L^2(\R_+;H),X))$.
\end{enumerate}
The equivalence of (3) and (3)$'$ has been proved in Proposition \ref{prop:alphaplus}.
\end{proof}

By Theorem \ref{thm:Xlq}(1), the family $\mathscr{T}_{\mathscr{K}^2}$ is $\ell^{1}$-bounded if $p/2>1$ and the dual of
$X^2$ is an HL lattice (the monotonically completeness assumption is automatically
satisfied for dual Banach lattices by \cite[Proposition 2.4.19]{MN}); recall that
$\wt{\mathscr{K}} = \{k\in L^1(\R): \ |k * f| \le \wt M|f|$ for all simple $f\}$.
Thus we have proved our main result for the stochastic convolution operators $S_k$:

\begin{theorem}\label{thm:Rbddstoch}
Let $X$ be a Banach lattice with type $2$
and suppose that the dual of its $2$-concavification $X^2$
is an HL lattice.
For all Hilbert spaces $H$ and all $p\in (2, \infty)$, the family
of stochastic convolution operators
$$ \{S_k^H: k^2 \in \wt{\mathscr{K}}\} $$
is $R$-bounded from
$L^{p}_\mathscr{F}(\R_+\times\O;\g(H,X))$ to $L^{p}(\R_+\times\O;X)$.
\end{theorem}

Recall that a sufficient condition for $X^2$ to be an HL space is that $X^2$ is a UMD Banach function space
(see Theorem \ref{thm:Xlq}(2)).

Note that if $k\in W^{1,1}_{\rm loc}(\R_+)$
satisfies $\lim_{t\to\infty} k(t) = 0$ and
$ \int_0^\infty \sqrt{t}|k'(t)|\,dt <\infty,$ then
\begin{align*}
 \int_0^\infty t |(k^2)' (t)| \,dt
& = \int_0^\infty 2t |k'(t)k(t)| \,dt
 \\ & = \int_0^\infty 2t\Big|k'(t)\int_t^\infty k'(s)\,ds \Big|\,dt
 \\ & \le \int_0^\infty 2\sqrt{t}|k'(t)| \int_t^\infty \sqrt{s}  |k'(s)|\,ds  \,dt
 \\ & \le 2\Big(\int_0^\infty \sqrt{t}|k'(t)|\,dt \Big)^2
\end{align*}
and therefore $k^2\in \wt{\mathscr{K}}$ by Propositions \ref{prop:Rbddkernelcond} and \ref{prop:K1}.
This motivates the following definition:

Let $\mathscr{S}$ be the class of all $k\in W^{1,1}_{{\rm loc}}(\R_+)$ such that
\[\lim_{t\to \infty}k(t) = 0 \ \ \text{and} \ \ \int_{\R_+} \sqrt{t} |k'(t)| \, dt\leq 1.\]
The $R$-boundedness of stochastic convolution with kernels $k\in \mathscr{S}$ was considered in \cite[Section 3]{NVW12a} in the case $X = L^q$ with $q\in [2, \infty)$.

Note that if $k\in \mathscr{S}$, the above estimate combined with Propositions \ref{prop:Rbddkernelcond}
and \ref{prop:K1} shows that $k^2\in\wt{\mathscr{K}}$. In particular, $k^2\in L^1(\R_+)$
and therefore $k\in L^2(\R_+)$.

\begin{corollary}\label{cor:main}
 Let $X$ be a Banach lattice with type $2$ and suppose that the dual of its $2$-concavification $X^2$
is an HL lattice.
For all Hilbert spaces $H$ and all $p\in (2, \infty)$, the family
of stochastic convolution operators
$$\{S_{k}^H: \ k\in \mathscr{S}\}$$
is $R$-bounded from
$L^{p}_\mathscr{F}(\R_+\times\O;\g(H,X))$ to $L^{p}(\R_+\times\O;X)$.
\end{corollary}

Examples of Banach lattices $X$ satisfying the conditions of the corollaries are the spaces $L^q(S)$
with $q\in [2,\infty)$ (we then have $X^2 = L^{q/2}(S)$).

\section{A counterexample\label{sec:failure}}

It has been an open problem for some time now whether
the family $$\{S_k^H: k^2\in\wt{\mathscr{K}}\} $$
considered in Theorem \ref{thm:Rbddstoch} is $R$-bounded from
$L^{p}_\mathscr{F}(\R_+\times\O;\g(H,X)$ to $L^{p}(\R_+\times\O;X)$ for all $2<p<\infty$ whenever
$X$ is a UMD Banach space with type $2$.  For UMD Banach lattices $X$ with type $2$, by Theorem \ref{thm:main1}
this question is equivalent to asking whether the family
$$\{T_{k^2}: k^2\in \wt{\mathscr{K}}\}
$$ is $\ell^{1}$-bounded
 on $L^{p/2}(\R_+;X^2)$ for any UMD Banach lattice
$X$ of type $2$. Here we will prove that this is not the case by showing that the space
$$ X = \ell^2(\ell^{4})$$
provides a counterexample;
for this space we have $X^2 = \ell^1(\ell^{2})$ and thus $(X^2)^* = \ell^\infty(\ell^{2})$.

Recalling that $\ell^\infty$ has the HL property, the following result comes somewhat as a surprise:
\begin{proposition}
The space $\ell^\infty(\ell^{2})$ fails the HL property.
\end{proposition}
\begin{proof} The proof is a refinement of the argument in \cite[Remark 2.9]{GMT1}.
Fix an integer $N\ge 1$. Let $f \in L^2(\R;\ell^\infty(\ell^2))$ be defined the coordinate functions
$$(f(t)_k)_j = \one_{(0,1]}(t)\one_{(2^{-j},2^{-j+1}]}(t-k2^{-N});$$
the indices $k$ and $j$ stand for the coordinates in $\ell^\infty$ and $\ell^2$,
respectively. Then $\n (f(t))_k\n_{\ell^2} = 1$ for all $t\in (0,1]$,
so $\n f\n_{L^2(\R;\ell^\infty(\ell^2))} = 1$. On the other hand
for $1\le j\le N$ and $\tau\in (k2^{-N},(k+1)2^{-N}]$ with $1\le k\le 2^N-1$ we have
\begin{align*} \wt M (f_k)_j(\tau)
& = \sup_{r>0} \frac1{2r} \Big|\int_{\tau-k2^{-N}-r}^{\tau-k2^{-N}+r} \one_{(2^{-j},2^{-j+1}]}(t)\,dt\Big|
\\ & \ge \frac1{2^{-j+2}} \Big|\int_{\tau-k2^{-N}-2^{-j+1}}^{\tau-k2^{-N}+2^{-j+1}} \one_{(2^{-j},2^{-j+1}]}(s)\,ds\Big|
\ge 2^{j-2} \cdot 2^{-j} = \frac14,
\end{align*}
so
$$ \n\wt M f(t)\n_{\ell^\infty(\ell^2)}^2 \ge  \sum_{j=1}^N (\wt M (f_k)_j(t))^2 \ge \frac{N}{16}, \quad t\in (2^{-N},1).$$
Hence $\n\wt M\n_{L^2(\R;\ell^\infty(\ell^2))} \ge \sqrt{N}/4 (1-2^{-N})^{1/2}$,
which tends to $\infty$ as $N\to \infty$.
\end{proof}

\begin{theorem}\label{thm:main3}
For any $1<p<\infty$, the family
$ \wt{\mathscr{T}} = \{T_{k}: k\in\wt{\mathscr{K}}\}$
fails to be $\ell^{1}$-bounded on $L^{p}(\R_+;\ell^1(\ell^{2}))$. As a consequence, for
any $2<p<\infty$ the family $$\{S_k^H: k^2\in \wt{\mathscr{K}}\} $$
fails to be $R$-bounded from
$L^{p}_\mathscr{F}(\R_+\times\O;\g(H,\ell^2(\ell^{4}))$ to $L^{p}(\R_+\times\O;\ell^2(\ell^{4}))$.
\end{theorem}

\begin{proof}
By a duality argument, it suffices to show that $\wt{\mathscr{T}}$ fails to be
$\ell^{\infty}$-bounded on $L^{p}(\R_+;\ell^\infty(\ell^2))$.
As $\ell^\infty(\ell^2)$ fails HL, the latter follows from Proposition \ref{prop:Xlinfty}.
\end{proof}

\noindent
{\em Acknowledgment} --
We thank Tuomas Hyt\"onen for his kind permission to present
his short proof of Proposition \ref{prop:Rbddkernelcond} here.
We thank the anonymous referee for carefully reading and providing
helpful comments.

\def\polhk#1{\setbox0=\hbox{#1}{\ooalign{\hidewidth
  \lower1.5ex\hbox{`}\hidewidth\crcr\unhbox0}}} \def\cprime{$'$}

\end{document}